\documentclass[a4paper,12pt]{article}
\usepackage[utf8]{inputenc}

\usepackage{amsmath}
\usepackage{amssymb}
\usepackage{amsthm}

\usepackage{cite}

\usepackage{booktabs}

\usepackage{esint}
\usepackage{color}
\usepackage{cite}
\usepackage{graphicx}
\usepackage{listings}
\usepackage{multirow}
\usepackage{tikz}
\usepackage{pgfplots}
\pgfplotsset{compat=1.16}
\usetikzlibrary{external}
\tikzsetfigurename{figure_}
\tikzset{external/system call={pdflatex \tikzexternalcheckshellescape -interaction=batchmode -jobname "\image" "\texsource"; rm "\image".dpth "\image".log "\image".nlo "\image".atfi "\image".spl}}

\newif\ifusetikz
\usetikzfalse

\title{On the solution existence for collocation discretizations of  time-fractional subdiffusion equations}
\author{Sebastian Franz\footnote{Institute of Scientific Computing, Technische Universit\"at Dresden, Germany.
          \mbox{e-mail}: sebastian.franz@tu-dresden.de,
          orcidID: {0000-0002-2458-1597}} \and 
        Natalia Kopteva\footnote{Department of Mathematics and Statistics, University of Limerick, Ireland.
          \mbox{e-mail}: natalia.kopteva@ul.ie,
          orcidID: {0000-0001-7477-6926}}}
\date{\today}

\makeatletter
\renewcommand*\env@matrix[1][r]{\hskip -\arraycolsep
  \let\@ifnextchar\new@ifnextchar
  \array{*\c@MaxMatrixCols #1}}
\makeatother
\definecolor{pass}{rgb}{0,0,0.7}
\newcommand{\pt}{\partial}
\newcommand{\R}{\mathbb{R}}
\newcommand{\pmtrx}[1]{\ensuremath{\begin{pmatrix}#1 \end{pmatrix}}}
\newcommand{\vmtrx}[1]{\ensuremath{\begin{vmatrix}#1 \end{vmatrix}}}

\newcommand{\LL}{{\mathcal L}}
\renewcommand{\AA}{{\mathcal A}}

\definecolor{pass}{rgb}{0,0,0.8}
\definecolor{pass1}{rgb}{0,0,0.5}

\theoremstyle{plain}
\newtheorem{theorem}{Theorem}[section]
\newtheorem{lemma}[theorem]{Lemma}
\newtheorem{corollary}[theorem]{Corollary}
\newtheorem{remark}[theorem]{Remark}

\newcommand{\new}[1]{{\color{blue!30!black}#1}}

\begin{document}
  \maketitle

    \begin{abstract}
      Time-fractional parabolic equations with a Caputo time derivative of order $\alpha\in(0,1)$ are discretized
      in time using continuous collocation methods.
      For such discretizations, we give sufficient conditions for existence and uniqueness of their solutions. Two approaches are explored: the Lax-Milgram Theorem and the eigenfunction expansion.
      The resulting sufficient conditions, which involve certain $m\times m$ matrices (where $m$ is the order of the collocation scheme),  are verified
      both
       analytically, for all $m\ge 1$ and all sets of collocation points,  and computationally, for all $ m\le 20$.
      The semilinear case is also addressed.
    \end{abstract}

    

  \section{Introduction}\label{sec:intro}

  We are interested in continuous collocation discretizations in time for
  time-fractional parabolic problems, of order $\alpha\in(0,1)$, of the form
  \begin{equation}\label{eq:problem}
    \pt_t^{\alpha}u+\LL u=f(x,t)\qquad\mbox{for}\;\;(x,t)\in\Omega\times(0,T],
  \end{equation}
  also known as time-fractional subdiffusion equations.
  This equation is posed in a bounded Lipschitz domain  $\Omega\subset\R^d$ (where $d\in\{1,2,3\}$),
  subject to an initial condition $u(\cdot,0)=u_0$ in $\Omega$, and the boundary condition $u=0$ on $\pt\Omega$ for $t>0$.
  The spatial operator $\LL$  is a linear second-order elliptic operator:
  \begin{equation} \label{LL_def}
    \LL u := \sum_{k=1}^d \Bigl\{-\pt_{x_k}\!(a_k(x)\,\pt_{x_k}\!u) + b_k(x)\, \pt_{x_k}\!u \Bigr\}+c(x)\, u,
  \end{equation}
  with sufficiently smooth coefficients $\{a_k\}$, $\{b_k\}$ and $c$ in $C(\bar\Omega)$, for which we assume that $a_k>0$ in $\bar\Omega$,
  and also either $c\ge 0$ or $c-\frac12\sum_{k=1}^d\pt_{x_k}\!b_k\ge 0$.
  We also use $\partial_t^\alpha$, the Caputo fractional derivative in time \cite{Diet10}, defined
  for $t>0$, by
  \begin{equation}\label{eq:CaputoEquiv}
    \pt_t^{\alpha} u := 
    \frac1{\Gamma(1-\alpha)} \int_{0}^t(t-s)^{-\alpha}\, \partial_s\new{u}(\cdot, s)\, ds,
  \end{equation}
  where $\Gamma(\cdot)$ is the Gamma function, and $\pt_t$ denotes the classical first-order partial derivative in~$t$.

  In two recent publications~\cite{FrK23, FrK23ENUMATH}, we considered high-order continuous collocation discretizations for solving \eqref{eq:problem}
  in the context of a-posteriori error estimation and
  reliable adaptive time stepping algorithms.
  It should be noted that despite a substantial literature on the a-priori error bounds for problem of type~\eqref{eq:problem}, both
 on uniform
  and graded temporal meshes---see, e.g., \cite{JLZ19,KopMC19,Kopteva21,Kopteva_Meng,Liao_etal_sinum2018,SORG17} and references therein---%
the a-priori error analysis of the collocation methods appears very problematic. 
By contrast, in combination with
the adaptive time stepping algorithm (based on the theory in \cite{Kopteva22}),
high-order collocation discretizations (of order up to as high as 8),
were shown to
yield reliable computed solutions and attain optimal convergence rates  in the presence of solution singularities.

  %
  Note also that the a-posteriori error analysis in \cite{Kopteva22}, as well as in related papers \cite{FrK23,FrK23ENUMATH,Kopt_Stynes_apost22},
  is carried out under the assumption that there exists a computed solution, the accuracy of which is being estimated (as well as there exists a solution
  of the original problem); in other words, it was assumed that the discrete problem was well-posed at each time level.
  While this assumption is immediately satisfied by the L1 method (as well as by a few L2 methods,
  which may be viewed as generalizations of multistep methods for ordinary differential equations),
  the well-posedness of high-order collocation schemes appears to be an open problem, to which we devote this article.

  Consider a continuous collocation method of order $m\ge 1$ for \eqref{eq:problem}
  (see, e.g.,~\cite{Brunner04}), associated with an arbitrary temporal grid $0=t_0<t_1<\dots<t_M=T$.
  We require our computed solution $U$ to be continuous in time on $[0,T]$,  a polynomial of degree $m$ in time on each $[t_{k-1},t_k]$, and satisfy
  \begin{equation} \label{Col_method}
    (\partial_t^\alpha+\LL)\,U(x, t^\ell_k)=f(x, t_k^\ell)\qquad \mbox{for}\;\; x\in\Omega,\;\;  \ell\in\{1,\ldots, m\},\;\;k=1\ldots, M,
  \end{equation}
  subject to the initial condition $U^0:=u_0$, and the boundary condition $U=0$ on $\pt \Omega$.
  Here a subgrid of collocation points $\{t_k^\ell\}_{\ell\in\{0,\dots,m\}}$ is used on each time
  interval $[t_{k-1},t_k]$ with $t_k^\ell:=t_{k-1}+\theta_\ell\cdot(t_k-t_{k-1})$, and
  \new{$0=\theta_0<\dots<\theta_m\leq1$.} 

  Strictly speaking, \eqref{Col_method} gives a semi-discretization of problem \eqref{eq:problem} in time,
  while applying any standard finite difference or finite element spatial discretization to \eqref{Col_method}
  will yield a so-called full discretization. The latter can typically be described by a version of
  \eqref{Col_method}, with $\Omega$, $\LL$, and $f$ respectively replaced by the corresponding  set $\Omega_h$ of
  interior mesh nodes, the corresponding discrete operator $\LL_h$, and the appropriate discrete right-hand side $f_h$.

  For $m=1$ the above collocation method is identical with the L1 method. In fact, then \eqref{Col_method} involves one elliptic equation at each time level $t_k$,
  with the elliptic operator $\LL+ \tau_k^{-\alpha}/\Gamma(2-\alpha)$,
  where $\tau_k:=t_k-t_{k-1}$ (or the corresponding discrete spatial operator $\LL_h+\tau_k^{-\alpha}/\Gamma(2-\alpha)$).
  Hence, using a standard argument (see, e.g., \cite[Lemma~2.1(i)]{kopteva_semilin}) one concludes that
  $c(\cdot)+\tau_k^{-\alpha}/\Gamma(2-\alpha)\ge 0$ in $\Omega$, $\forall\,k\ge 1$, is sufficient for existence and uniqueness of a solution $U$ of \eqref{Col_method}.

  Once $m\ge 2$, one needs to solve a system of $m$ coupled elliptic equations at each time level (or a spatial discretization of such a system), so the well-posedness of such systems is less clear.
  By contrast, one can easily ensure the well-posedness in the case of a time-fractional ordinary differential equation
  (i.e. if $\LL u$ in \eqref{eq:problem} is replaced by $c(t)u$) by simply choosing the time steps sufficiently small (see section~\ref{ssec_ODE}).
  This approach is, however, not applicable to subdiffusion equations (as discussed Remark~\ref{rem_inapplble}), since
  the eigenvalues of second-order elliptic operators are typically unbounded.
  Hence, in this paper we shall explore alternative approaches to establishing the well-posedness of~\eqref{Col_method}.

  The paper is organised as follows. We start our analysis by
  deriving, in section~\ref{sec_matrix}, a matrix representation of the collocation scheme \eqref{Col_method}.
  Next, in section~\ref{sec_LMil},
  sufficient conditions for the well-posedness of \eqref{Col_method} are obtained
  by means of the Lax-Milgram Theorem, with a detailed application for \mbox{$m=2$}.
  Furthermore, alternative sufficient conditions for the well-posedness of \eqref{Col_method}
  are formulated in section~\ref{sec_lambda}
  in terms of an eigenvalue test for a certain $m\times m$ matrix.
This test is then used in section~\ref{ssec_analy} to prove the existence of collocation solutions for all $m\ge 1$
and all sets of collocation points $\{\theta_\ell\}_{\ell=1}^m$; see our main result, Theorem~\ref{the_cor}.
  This test can also be applied computationally, as illustrated in section~\ref{ssec_comp}
  for all $ m \le 20$.
  Finally, in section \ref{sec_semi}, we present extensions to the semilinear case.
\smallskip

  \textit{Notation.}
  We use the standard spaces $L_2(\Omega)$ and   $H^1_0(\Omega)$, the latter denoting the space of functions in the Sobolev space $W_2^1(\Omega)$ vanishing on $\pt\Omega$, as well as its dual space $H^{-1}(\Omega)$.
  The notation of type $(L_2(\Omega))^m$ is used for the space of vector-valued functions
  $\Omega\rightarrow \R^m$ with all vector components in $L_2(\Omega)$, and the norm induced by $\|\cdot\|_{L_2(\Omega)}$ combined with any fixed norm in $\R^m$.

  \section{Matrix representation of the collocation scheme \eqref{Col_method}}\label{sec_matrix}
  As we are interested only in the existence and uniqueness of a solution of the collocation scheme~\eqref{Col_method}, it suffices to consider only the first time interval $(0,t_1)$, i.e. $k=1$ in \eqref{Col_method}, and with a homogeneous initial condition. Furthermore,
  \new{setting $\tau:=\tau_1$,}
  we shall rescale the latter problem to the time interval $(0,1)$. Hence,  we are ultimately interested in the well-posedness of the problem
  \begin{equation} \label{Col_method1}
    (\partial_t^\alpha+\tau^\alpha \LL)\,U(x, \theta_\ell)=F(x, \theta_\ell)\qquad \mbox{for}\;\; x\in\Omega,\;\;  \ell\in\{1,\ldots, m\},
  \end{equation}
  where $U$ is a polynomial of degree $m$ in time for $t\in[0,1]$
  such that $U(\cdot, 0)=0$ and
  $U=0$ on $\pt\Omega$.

  \new{
  To relate \eqref{Col_method} on any time interval $(t_{k-1},t_k)$ to \eqref{Col_method1},
  set $\theta:=(t-t_{k-1})/\tau_k\in(0,1)$ (so $t=t_k^\ell=t_{k-1}+\theta_\ell\cdot\tau_k$
  corresponds to $\theta=\theta_\ell$)
   and $\hat U(\cdot,\theta):=U(\cdot,t)-U(\cdot,t_{k-1})$.
  Then a calculation shows that $ \LL U(\cdot,t)=\LL \hat U(\cdot,\theta)-\LL U(\cdot,t_{k-1})$
  and $\partial_t^\alpha U(\cdot,t)=\tau_k^{-\alpha}\,\partial_\theta^\alpha \hat U(\cdot,\theta)+F_0(\cdot, t)$,
  where $  F_0(\cdot, t):=\{\Gamma(1-\alpha)\}^{-1}\int_0^{t_{k-1}}(t-s)^{-\alpha}\pt_s U(\cdot,s)\,ds$.
  Now,
  \eqref{Col_method} on  $(t_{k-1},t_k)$ is equivalent to
  \[
    \partial_\theta^\alpha \hat U(\cdot,\theta_\ell)+\tau_k^{\alpha}\,\LL \hat U(\cdot,\theta_\ell)=F(\cdot,\theta_\ell),
    \;\;\mbox{where}\;
    F(\cdot,\theta):=\tau_k^{\alpha}[f(\cdot, t)+\LL U(t_{k-1})-F_0(\cdot, t)].
  \]
  Finally, with a slight abuse of notation, replacing $\hat U$ in the above by a $U(\cdot,t)$, and also
  using the notation  $\tau:=\tau_k$ for the current time step to further simplify the presentation,
  one gets \eqref{Col_method1}.}

  \new{Furthermore, without loss of generality, our analysis will be presented only for the case $\theta_m=1$.

  \begin{remark}[Case $\theta_m<1$]\label{rem_theta_1}
  All our findings for the case $0=\theta_0<\dots<\theta_m=1$ immediately apply to the general case
  $0=\theta_0<\dots<\theta_m\le 1$.
  Indeed, if $\theta_m<1$, it suffices to rescale \eqref{Col_method1} from the interval $(0,\theta_m)$
  to $(0,1)$, using $\hat U(\cdot, t):=U(\cdot, t/\theta_m)$ and $\hat F(\cdot, t):=\theta_m^\alpha F(\cdot, t/\theta_m)$.
Then $\pt_t^\alpha  U(\cdot, t)=\theta_m^{-\alpha}\pt_t^\alpha \hat U(\cdot, t)$, so  \eqref{Col_method1} becomes
  $$
  (\partial_t^\alpha+\hat\tau^\alpha \LL)\,\hat U(x, \hat\theta_\ell)= F(x, \hat\theta_\ell),\quad \mbox{where}\;\;
  \hat\tau:=\theta_m\tau_k\;\;\mbox{and}\;\;
  \{\hat\theta_\ell:=\theta_\ell/\theta_m\}_{\ell=1,\ldots, m}\,.
  $$
  The above is clearly a problem of type \eqref{Col_method1} with $0=\hat\theta_0<\dots<\hat\theta_m=1$.
  \end{remark}
  }

  A solution $U$ of \eqref{Col_method1} is a polynomial in time, equal to  $0$ at $t=0$, so it can be represented
  using the basis $\{t^j\}_{j=1}^m$ as
  \begin{equation}\label{matr1}
    U(x,t)=\sum_{j=1}^m v_j(x)\,t^j = \Bigl(t,\,t^2,\,\cdots,\,t^m\Bigr)\pmtrx{v_1(x)\\v_2(x)\\\vdots\;\;\\v_m(x)}=:T(t)\, \vec{V}(x).
  \end{equation}
  Here we used the standard linear algebra multiplication, and, for convenience, we  highlight column
  vectors in such evaluations with an arrow (as in $\vec{V}$).
  Other column vectors of interest are
  $\vec{\theta}:=\big(\theta_1,\theta_2,\cdots,\theta_m\big)^{\!\!\top}$,
  and also $\vec{U}(x):=U(x,\vec{\theta})$ and $\vec{F}(x):=F(x,\vec{\theta})$
  (where a function is understood to be applied to a vector argument elementwise).
  Thus, $\vec{U}$ can be represented using
  the Vandermonde-type matrix $W$ as follows:
  \begin{equation}\label{matr2}
    W:=\pmtrx{T(\theta_1)\\T(\theta_2)\\\vdots\;\;\;\\T(\theta_m)}=
       \pmtrx{\theta_1 & \theta_1^2 & \cdots &\theta_1^m\\
              \theta_2 & \theta_2^2 & \cdots &\theta_2^m\\
              \vdots\ &\vdots\ && \vdots\ \\
              \theta_m & \theta_m^2 &\cdots & \theta_m^m}
    \; \Rightarrow\;
    \vec{U}(x)=U(x,\vec{\theta})
    =W \vec{V}(x).
  \end{equation}
  Additionally, we shall employ two diagonal matrices
  \begin{equation}\label{matr3}
    D_1={\rm diag}(\theta_1^{-\alpha},\cdots,\theta_m^{-\alpha}), \;\; D_2:={\rm diag}(c_1,\cdots,c_m),\;\;
    c_j=\frac{\Gamma(j+1)}{\Gamma(j+1-\alpha)}.
  \end{equation}
  Note that the coefficients of $D_2$ appear in
  $\pt_t^\alpha t^j=c_j\, t^{j-\alpha}$. Note also that,
  in view of \new{$0<\theta_1<\ldots<\theta_m$},
  the Vandermonde-type matrix $W$ has a positive determinant
  $\det (W)=\bigl(\prod_{j=1}^m \theta_j\bigr)\prod_{1\le i<j\le m}(\theta_j-\theta_i)>0$, so it is invertible.

  With the above notations, \eqref{Col_method1} is equivalent to
  \begin{equation}\label{temp_matrix}
    \partial_t^\alpha U(x, \vec{\theta})+\tau^\alpha \LL U(x, \vec{\theta})=\vec{F}(x),
    \;\;\mbox{where}\;\;
    \LL U(x, \vec{\theta})=\LL\vec{U}(x)=\LL W \vec{V}(x),
  \end{equation}
  and where we still need to evaluate $\partial_t^\alpha U(x,\vec{\theta})$. For the latter, one has
  \[
    \pt_t^\alpha U(x,t)
    =\sum_{j=1}^m v_j(x)\,\pt_t^\alpha t^j
    =\sum_{j=1}^m c_j\, t^{j-\alpha}\,v_j(x)
      = t^{-\alpha}\,T(t)\,D_2\,\vec{V}(x),
  \]
  so
  \[
    \partial_t^\alpha U(x, \vec{\theta})=D_1WD_2\,\vec{V}(x).
  \]
  Substituting this in \eqref{temp_matrix}, we arrive at the following result.
  \begin{lemma}\label{lem_matr}
    With the notations \eqref{matr1},\,\eqref{matr2},\,\eqref{matr3}, the collocation scheme \eqref{Col_method1} is equivalent to
    \begin{equation}\label{matrix_col}
      \bigl(D_1 WD_2+\tau^\alpha\LL W\bigr)\,\vec{V}(x) = \vec{F}(x).
    \end{equation}
  \end{lemma}
\smallskip
  \begin{remark}[{Case $\LL=\LL(t)$}]\label{rem_matrix_L_t}
    The matrix representation \eqref{matrix_col} is valid only if the spatial operator $\LL$ is independent of $t$.
    If $\LL=\LL(t)$, then $\LL$ in \eqref{matrix_col} should be replaced by
    the diagonal matrix ${\rm diag}(\LL(\theta_1),\LL(\theta_2),\cdots,\LL(\theta_m))$.
  \end{remark}

\subsection{Well-posedness for the case $\LL=c(t)$ without spatial derivatives}\label{ssec_ODE}
  In the case of a time-fractional ordinary differential equation of type~\eqref{Col_method} with $\LL$ replaced by $c(t)$,
  one needs to investigate the well-posedness of a version of \eqref{Col_method1}
  with $\LL$ replaced by $\hat c(\theta_{\ell})$, where $\hat c(t):=c((t-t_{k-1})/\tau_k)$.
  Now, in view of Remark~\ref{rem_matrix_L_t}, the matrix representation \eqref{matrix_col} reduces to
  \begin{equation}\label{matrix_col_ODE}
    \bigl(D_1 WD_2+\tau^\alpha D_c W\bigr)\,\vec{V} = \vec{F},
    \qquad
    D_c:={\rm diag}(\hat c(\theta_1),\cdots,\hat c(\theta_m)),
  \end{equation}
  where $\vec{V}$ and $\vec{F}$ are $x$-independent, as well as the corresponding solution $U=U(t)$ of \eqref{Col_method1}.
  Recall also that here the Vandermonde-type matrix $W$ has a positive determinant and is invertible, as well as the diagonal matrices $D_1$ and $D_2$.
  Hence,  the matrix equation \eqref{matrix_col_ODE} will have a unique solution if and only if the matrix
  $I+\tau^\alpha(D_1 WD_2)^{-1}D_c W$ is invertible. A well-known sufficient condition for this is
  \begin{equation}\label{tau_small}
    \tau^\alpha\|(D_1 WD_2)^{-1}D_c W\|<1,
  \end{equation}
  where $\|\cdot\|$ denotes any matrix norm.

  Thus, for any bounded $c(t)$, any $m\ge1$, and any set of collocation points $\{\theta_\ell\}_{\ell=1}^m$,
  one can choose a sufficiently small $\tau$ to guarantee the well-posedness of \eqref{matrix_col_ODE}.

  \begin{remark}\label{rem_inapplble}
    Importantly, the above argument is not applicable to the collocation scheme \eqref{Col_method1} with an elliptic operator $\LL$. For example, the simplest
    one-dimensional Laplace operator $\LL=-\partial_x^2$ in the domain $\Omega=(0,1)$ has the set of eigenvalules
    $\{\lambda_n=(\pi n)^2\}_{n=1}^\infty$ and the corresponding basis of orthonormal eigenfunctions.
    Hence, an application of the eigenfunction expansion to the solution vector $\vec V(x)$
    immediately implies that a version of \eqref{Col_method1} with $\LL$ replaced by $\lambda_n$
    needs to be well-posed $\forall\, n\ge 1$.
    Clearly, in this case one cannot satisfy a sufficient condition of type~\eqref{tau_small}, which would take the form
    \begin{equation}\label{tau_small_h}
      \tau^\alpha\,\lambda_n\,\|(D_1 WD_2)^{-1} W\|<1 \qquad\forall n\ge1,
    \end{equation}
    in view of $\lambda_n\rightarrow \infty$ as $n\rightarrow \infty$.
    Hence, in this paper we explore alternative approaches to establishing the well-posedness of \eqref{Col_method1}.

    More generally, the eigenvalues of self-adjoint elliptic operators $\LL$ with variable coefficients
    in $\Omega\subset\R^d$ (where $d\in\{1,2,3\}$)
    are also real and exhibit a similar behaviour $\lambda_n\rightarrow \infty$ as $n\rightarrow \infty$
    \cite[section~6.5.1]{evans}, so the sufficient condition \eqref{tau_small_h} simply cannot be satisfied in this case.

    The situation somewhat improves once \eqref{Col_method1} is discretized in space, as then the operator $\LL$ is replaced by a finite-dimensional discrete operator $\LL_h$, the eigenvalues of which are bounded, so \eqref{tau_small_h} can now be imposed, albeit it may be very restrictive. For example,
    for the standard finite difference discretization $\LL_h$ of $\LL=-\partial_x^2$ in the domain $\Omega=(0,1)$,
    it is well-known that $\max_{n}\lambda_n\approx 4\cdot DOF^{2}$, where $DOF$ is the number of degrees of freedom in space. Furthermore, the maximal $\lambda_n$ may simply be unavailable for spatial discretizations of variable-coefficient elliptic operators in complex domains on possibly strongly non-uniform spatial meshes with local mesh refinements.%
  \end{remark}

\section{Well-posedness by means of the Lax-Milgram Theorem}\label{sec_LMil}
  Let $D_0={\rm diag}(d_1,\cdots, d_m)$ be an arbitrary diagonal matrix with strictly positive diagonal elements.
  Note that a left multiplication of \eqref{matrix_col} by $W^{\top}D_0$ (where both $W$ and $D_0$ are invertible) produces an equivalent system.
  Taking the inner product $\langle\cdot,\cdot\rangle$ of the latter with an arbitrary $\vec V^*\in (H^1_0(\Omega))^m$, in the sense of the $(L_2(\Omega))^m$ space, yields a variational formulation of \eqref{matrix_col}: Find $\vec V \in \new{(H^1_0(\Omega))^m}$ such that
  \begin{equation}\label{AA_from}
    \AA (\vec V, \vec V^*)=
    \langle W^{\top}D_0 D_1 WD_2\,\vec{V}, \vec V^* \rangle+\tau^\alpha\langle W^{\top}D_0 \LL W\vec V,\vec V^* \rangle = \langle W^{\top}D_0\vec{F},\vec{V}^*\rangle
  \end{equation}
  $\forall\, \vec V^*\in \new{(H^1_0(\Omega))^m}$.
  The well-posedness of this problem can be established using the Lax-Milgram Theorem as follows.

  \begin{theorem}\label{the_LaxM}
    Suppose that the bilinear form $a(v,w):=\langle \LL v, w\rangle$ on the space $H^1_0(\Omega)$ is bounded and coercive.
    Additionally, suppose that there exists a diagonal matrix $D$ with strictly positive diagonal elements such that
    the symmetric matrix $W^{\top}D WD_2+(W^{\top}D WD_2)^{\top}$ is positive-semidefinite.
    Then the bilinear \new{form} $\AA (\cdot, \cdot)$ on the space $(H^1_0(\Omega))^m$ is
    also bounded and coercive for $D_0:=D D_1^{-1}$. If, additionally, $\vec F\in (H^{-1}(\Omega))^m$, then
    there exists a unique solution to problem \eqref{AA_from}. Equivalently, if
    $\{F(\cdot, \theta_\ell)\}_{\ell=1}^m\in (H^{-1}(\Omega))^m$, then
    \eqref{Col_method1} has a unique solution $\{U(\cdot, \theta_\ell)\}_{\ell=1}^m$
    in $(H^1_0(\Omega))^m$.
  \end{theorem}

  \begin{proof}
    The boundedness and coercivity of $\AA (\cdot, \cdot)$ follow from the boundedness and coercivity of $a(\cdot,\cdot)$. In particular, for the coercivity, note that
    \begin{equation}\label{AA_from_new}
      \AA (\vec V, \vec V)=
      \langle W^{\top}D_0 D_1 WD_2\,\vec{V}, \vec V \rangle+\tau^\alpha
      \langle D_0 \LL W\vec V, W\vec V \rangle.
    \end{equation}
    For the final term here, recalling that $\vec U=W\vec V$, one gets
    \[
      \langle D_0 \LL W\vec V, W\vec V \rangle=\langle D_0 \LL \vec U, \vec U \rangle
      =\sum_{\ell=1}^m d_{\ell}\,a(U_\ell,U_\ell),
    \]
    where $\{d_\ell\}$ are strictly positive,
    so this term is coercive. Here we also used the equivalence of the norms of $\vec V$ and $\vec U=W\vec V$
    in $(H^1_0(\Omega))^m$ (as the matrix $W$ is invertible).

    Thus, for the coercivity of $\AA (\cdot, \cdot)$, it remains to note that
    the quadratic form generated by the matrix $W^{\top}D_0 D_1 WD_2=W^{\top}D WD_2$ is positive-semidefinite,
    as the same quadratic form is also generated by the symmetric
    positive-semidefinite matrix $\frac12[W^{\top}D WD_2+(W^{\top}D WD_2)^{\top}]$.

    Finally, the existence and uniqueness of a solution to \eqref{AA_from}, or, equivalently, to
    \eqref{Col_method1}, immediately follow by the Lax-Milgram Theorem (see, e.g., \cite[section~2.7]{BrenScott}).
  \end{proof}

  \begin{remark}[{Case $\LL=\LL(t)$}]\label{rem_matrix_L_t_coerc}
    Theorem~\ref{the_LaxM} remains valid for a time-dependent $\LL=\LL(t)$ as long as $a(\cdot, \cdot)$ is coercive $\forall\,t>0$.
    The proof applies to this more general case with minimal changes, in view of Remark~\ref{rem_matrix_L_t}.
  \end{remark}

  \begin{remark}[{Coercivity of $a(v,w):=\langle \LL v, w\rangle$}]\label{rem_a_coercive}
    One can easily check that for the coercivity of the bilinear form $a(v,w):=\langle \LL v, w\rangle$
    it suffices that the coefficients of $\LL$ in~\eqref{LL_def}  satisfy  $a_k>0$  and
    $c-\frac12\sum_{k=1}^d\pt_{x_k}\!b_k\ge 0$ in $\bar\Omega$.
  \end{remark}

  \begin{corollary}[Quadratic collocation scheme]\label{cor_m_2}
    If $m=2$, with the collocation points $\{\theta_1, 1\}$, for any
    $\theta_1\le\theta^*:=1-\frac12\alpha$,
    there exists a diagonal matrix $D$ required by Theorem~\ref{the_LaxM}.
    Hence, under the other conditions of this theorem, there exists a unique solution
    to \eqref{Col_method1} in $(H^1_0(\Omega))^2$.
  \end{corollary}

  \begin{proof}
    Recall that for $m=2$, one has $0<\theta_1<\theta_2=1$.
    Let $D:={\rm diag}(1,p)$ with some $p>0$.
    Without loss of generality, we shall replace $D_2$, defined in \eqref{matr3}, by
    its normalized version $\hat D_2:=D_2/c_1={\rm diag}(1,C)$, where
    \begin{equation}\label{eq_C}
    C:=\frac{c_2}{c_1}=\frac{\Gamma(3)}{\Gamma(3-\alpha)}\frac{\Gamma(2-\alpha)}{\Gamma(2)}
    =\frac{2}{2-\alpha}\in(1,2).
    \end{equation}
    Next, using the definition of $W$ in~\eqref{matr2},
    and setting $\theta:=\theta_1$ to simplify the presentation,
    one gets
    \[
      W^{\top}D W\hat D_2
        = \pmtrx{\theta\,\,   &1\,\,\\
                 \theta^2 & 1^2}
          \cdot
          \pmtrx{1&0\\0&p}
          \cdot
          \pmtrx{\theta & \theta^2\\
                 1 & 1^2}
          \cdot
          \pmtrx{1&0\\
                 0&C}
        = \pmtrx{\theta^2+p & \ C(\theta^3+p)\\
                 \theta^3+p & C(\theta^4+p)}.
    \]
    For the symmetric matrix $W^{\top}D W\hat D_2+(W^{\top}D W\hat D_2)^\top$ to be positive-semidefinite,
    by the Sylvester’s criterion, we need its derterminant to be non-negative (while its diagonal elements
    are clearly positive). Thus we need to check that, for some $p>0$,
    \[
      \det \pmtrx{2(\theta^2+p) & \ (1+C)(\theta^3+p)\\
                  (1+C)(\theta^3+p) & 2C(\theta^4+p)}\ge 0.
    \]
    This is equivalent to
    \[
      (1+C)^2\, (\theta^3+p)^2\le 4C\,(\theta^2+p)\,(\theta^4+p).
    \]
    Dividing the above by $4C\, \theta^6$, and using the notation $P:=p/\theta^{3}$
    and $\psi(s):=s+s^{-1}-2=\frac{(s-1)^2}{s}$, yields
    \[
      \bigl[1+{\textstyle\frac14}\psi(C)\bigr]\, (1+P)^2\le (\theta^{-1}+P)\,(\theta+P),
    \]
    where we also used the observation that $\frac{(1+C)^2}{4C}=1+\frac{(C-1)^2}{4C}=1+\frac14\psi(C)$.
    Next, a calculation yields $(\theta^{-1}+P)\,(\theta+P)=(1+P)^2 +P\psi(\theta)$, so one gets
    \[
      (1+P)^2+{\textstyle\frac14}\psi(C)\,(1+P)^2 \le (1+P)^2 +P\psi(\theta),
    \]
    or
    \[
      (1+P)^2\le 4P\,\frac{\psi(\theta)}{\psi(C)}.
    \]
    Setting $\kappa:={\psi(\theta)}/{\psi(C)}$, the above is equivalent to
    $(P-1)^2\le 4(\kappa-1)P$. For any positive $P$ to satisfy this condition, it is necessary and sufficient
    that $\kappa\ge 1$ (in which case, one may choose $P=1$, which corresponds to
    $p=\theta^3$)
    Thus for the existence of a diagonal matrix $D$ required by Theorem~\ref{the_LaxM}, we need to impose
    ${\psi(\theta)}\ge{\psi(C)}$.
    Finally, note that $C> 1$, while $\psi(\theta)=\psi(\theta^{-1})$ and $\theta^{-1}>1$. As $\psi(s)$
    is increasing for $s>1$, we impose
    $\theta^{-1}\ge C$, or $\theta_1=\theta\le\theta^*=\frac{2-\alpha}2=1-\frac12\alpha$.
  \end{proof}

    \begin{remark}[Case $\theta_2<1$]
    \new{In view of Remark~\ref{rem_theta_1}, Corollary~\ref{cor_m_2}
     applies to the case
    $m=2$, with the collocation points $0=\theta_0<\theta_1<\theta_2\le 1$
    under the condition $\theta_1\leq\theta_2\left( 1-\frac{1}{2}\alpha \right)$.}
    \end{remark}

  \begin{corollary}[Galerkin finite element discretization]
    The existence and uniqueness results of Theorem~\ref{the_LaxM} and Corollary~\ref{cor_m_2}
    remain valid for a version of
    \eqref{AA_from} with $(H^1_0(\Omega))^m$ replaced by $(S_h)^m$, where $S_h$ is a finite-dimensional subspace of
    $H^1_0(\Omega)$.
    Equivalently, they apply to the corresponding spatial discretization of \eqref{Col_method1}.
  \end{corollary}

  For higher-order collocation schemes with $m\ge 3$, the construction of a suitable matrix $D$
  (as described in the proof of Corollary~\ref{cor_m_2})
  requires $m-1$ positive parameters, so such constructions are unclear for general $\alpha\in(0,1)$ and $\{\theta_\ell\}_{\ell=1}^m$.
  In this case, $D$ may still be constructed for particular values of $\alpha$ and $\{\theta_\ell\}$
  of interest (then semi-computational techniques, such as employed in section~\ref{ssec_comp}, may be useful).
  Furthermore, in the next section~\ref{sec_lambda}, we shall describe alternative approaches to establishing the well-posedness of~\eqref{Col_method1}, which can be easily applied for any $m$ of interest.

\section{Well-posedness by means of an eigenvalue test}\label{sec_lambda}

  Throughout this section, we shall assume that the time-independent operator $\LL$
  has a set of real positive \new{eigenvalues}
  $0<\lambda_1\le\lambda_2\le \lambda_3\le\ldots$ and a corresponding basis of orthonormal eigenfunctions
  $\{\psi_n(x)\}_{n=1}^\infty$.
  This assumption is immediately satisfied if $\LL$ in \eqref{LL_def} is
  self-adjoint (i.e. $b_k=0$ for $k=1\new{,}\ldots,d$) and $c\ge 0$; see, e.g., \cite[section~6.5.1]{evans}  for further details.
  (Even if $\LL$ is non-self-adjoint, the problem can sometimes,  as, e.g., for
  $\LL = -\sum_{k=1}^d \bigl\{\pt^2_{x_k} + (\pt_{x_k}\! B(x))\, \pt_{x_k} \bigr\}+c(x)$
  \cite[problem~8.6.2]{evans},
   be reduced to the self-adjoint case.)
 %

  The above assumption essentially allows us to employ the method of separation of variables.
  Thus, the well-posedness of problem \eqref{matrix_col} can be investigated using
  an eigenfunction expansion of its solution $\vec{V}(x)$.

  Recall \eqref{matrix_col} and rewrite it using $\vec U=W\vec V$ as
  \begin{equation}\label{matrix_col_U}
    \bigl(D_1 WD_2W^{-1}+\tau^\alpha\LL \bigr)\,\vec{U}(x) = \vec{F}(x).
  \end{equation}
  Next, using the eigenfunction expansions $\vec{U}(x)=\sum_{\new{n}=1}^\infty \vec{U}_n\, \psi_n(x)$
  and $\vec{F}(x)=\sum_{n=1}^\infty \vec{F}_n\, \psi_n(x)$, one reduces the above to
  the following matrix equations for vectors $\vec{U}_n$ $\forall\,n\ge1$:
  \begin{equation}\label{lambda_n_prob}
    \bigl(D_1 WD_2W^{-1}+\tau^\alpha\lambda_n\new{I} \bigr)\,\vec{U}_n = \vec{F}_n\,.
  \end{equation}
  For each $n$, we get an $m\times m$ matrix equation,  the well-posedness of which is addressed by
  the following elementary lemma.

  \begin{lemma}\label{lem_M}
    Set $M:=D_1 WD_2W^{-1}$ and $R(\lambda\,;M):=(M+\lambda\new{I})^{-1}$. Then
    \begin{equation}\label{resolvent}
      0<\| R(\lambda\,;M)\|\le C_M\;\;\;\forall\lambda\ge 0\quad  \mbox{and}\quad
      \lim_{\lambda\rightarrow\infty}\bigl\{\lambda\,\| R(\lambda\,;M)\|\bigr\}=1,
    \end{equation}
    with some positive constant $C_M=C_M(\alpha, m, \{\theta_\ell\})$
    (where  $\|\cdot\|$ denotes the  matrix norm induced by 
    any vector norm in $\R^m$),
    if and only if $M$ has no real negative eigenvalues.
  \end{lemma}

  \begin{proof}
    The final assertion in \eqref{resolvent} is obvious, in view of $R(\lambda\,;M)=\lambda^{-1}(\lambda^{-1}M+I)^{-1}\rightarrow \lambda^{-1}I$ as $\lambda\rightarrow \infty$.
    Next, if $M$ has a real negative eigenvalue $-\lambda^*$, then
    $M+\lambda^*\new{I}$ is a singular matrix, so
    $R(\lambda^*\,;M)$ is not defined, so \eqref{resolvent} is not true.
    If $M$ has no real negative eigenvalues, then $R(\lambda\,;M)$ is defined $\forall\,\lambda>0$.
    It is also defined for $\lambda=0$, as $M=D_1 WD_2W^{-1}$ is invertible (as discussed in section~\ref{sec_matrix}).
    Furthermore, $R(\lambda\,;M)\,\vec v=0$ if and only if $\vec v=0$, so $\| R(\lambda\,;M)\|>0$ $\forall\, \lambda\ge 0$. As $\| R(\lambda\,;M)\|$ is a continuous positive function of $\lambda$ on $[0,\infty)$, which decays as $\lambda\rightarrow \infty$, it has to be bounded by some positive constant $C_M$.
  \end{proof}

  As problem \eqref{matrix_col_U} (or, equivalently, \eqref{Col_method1}) is reduced to the set of problems~\eqref{lambda_n_prob},
  the above lemma immediately yields sufficient conditions for the well-posedness
  as long as $\vec F(x)$ allows an eigenfunction expansion, in which case there exists
  $\vec U(x)$ (the regularity of which follows from \eqref{resolvent}).
  Thus, we get the following result.

  \begin{theorem}\label{the_M_matrix}
    Suppose that the operator $\LL$ has a set of real positive \new{eigenvalues}
    and a corresponding basis of orthonormal eigenfunctions,
    and $\vec F(x)$ is sufficiently smooth.
    If, additionally, the matrix $M=D_1 WD_2W^{-1}$ has no real negative eigenvalues, then
    there exists a unique solution to problem \eqref{matrix_col_U}, or, equivalently,
    to problem \eqref{Col_method1}.
  \end{theorem}

  \begin{remark}[Full discretizations]
    The existence and uniqueness result of Theorem~\ref{the_M_matrix}, as well as
    its important corollary presented as Theorem~\ref{the_cor} below,
    remains valid for any spatial discretization of \eqref{matrix_col_U}, or, equivalently,
    of problem \eqref{Col_method1},
    as long as the corresponding discrete operator
    $\LL_h$ (which replaces $\LL$ in  \eqref{Col_method1} and, hence, \eqref{matrix_col_U}) has a set of real positive \new{eigenvalues}
    and a corresponding basis of orthonormal eigenvectors.

    This includes standard Galerkin finite element discretizations in space. Indeed, as shown, e.g., in \cite[chapter~6]{strang_fix} (by representing the eigenvalues via the Rayleigh quotient),
    not only all eigenvalues $\{\lambda_n^h\}$ of the corresponding discrete operator $\LL_h$ are real positive, but they also satisfy $\lambda_n^h\ge\lambda_n\ge\lambda_1>0$. The existence of
    an eigenvector basis follows from $\LL_h$ being self-adjoint, i.e. being associated with a symmetric real-valued matrix.

    Similarly, finite difference discretizations
    of spatial self-adjoint operators (such as discussed in \cite[section~4]{KopMC19})
    typically satisfy the discrete maximum principle (which immediately excludes real negative and zero eigenvalues), and are associated with symmetric real-valued matrices (which mimics $\LL$ being self-adjoint). Hence, all discrete eigenvalues are real positive, and there is a corresponding basis of eigenvectors. So Theorem~\ref{the_M_matrix} is applicable to such full discretizations as well.
  \end{remark}

  To apply the above theorem, one needs to check whether the  matrix
  $M=D_1 WD_2W^{-1}$ has no real negative eigenvalues.
  Clearly, $M$ is a fixed matrix of a relatively small size $m\times m$;
  however, it depends on $m$, $\alpha$, and the set of collocation points $\{\theta_\ell\}$.
  Below we shall investigate the spectrum of $M$,  first, analytically, and then using a semi-computational approach.
  Note also that
  in practical computations, the spectrum of $M$ may be easily checked computationally
  for any fixed  $m$, $\alpha$, and $\{\theta_\ell\}$ of interest.

\subsection{An analytical approach}\label{ssec_analy}

  To ensure that $M=D_1 WD_2W^{-1}$ has no real negative eigenvalues, one does not necessarily need to
  evaluate the eigenvalues of this matrix. Instead, the coefficients of the characteristic
  polynomial may be used. In particular, for a characteristic polynomial of the form
  $\sum_{j=0}^m (-\lambda)^j a_j$ to have no real negative roots, it suffices to check that $a_j>0$ $\forall\,j=0\new{,}\ldots,m$.
  (We also note  the stronger, and more intricate, Routh-Hurwitz stability criterion \cite[Vol. 2, Sec. 6]{Gantmacher59},
  which gives necessary and sufficient conditions for all roots of a given polynomial to have positive real parts.)
%

The above observation leads us to the main result of the paper on the existence and uniqueness of collocation solutions
for any $\alpha\in(0,1)$, any $m\ge 1$, and any set of collocation points $\{\theta_\ell\}_{\ell=1}^m$.

  \begin{theorem}\label{the_cor}
    For any $m\ge 1$, no eigenvalue of the matrix $M=D_1 WD_2W^{-1}$ is real negative. Hence,
    Theorem~\ref{the_M_matrix} applies
    for any $\alpha\in(0,1) $ and any set of collocation points $0<\theta_1 <\theta_2< \dots < \theta_m\new{\leq1}$.
  \end{theorem}
 \begin{proof}
Recalling the definitions of $W$, $D_1$, and $D_2$ from~\eqref{matr2} and~\eqref{matr3},
    let
    \begin{align*}
      M_\alpha:=D_1W  D_2
      &=\pmtrx{[c]
               c_1 \theta_1^{1-\alpha} & c_2 \theta_1^{2-\alpha} & \cdots & c_m \theta_1^{m-\alpha}\\[1pt]
               c_1 \theta_2^{1-\alpha} &   c_2 \theta_2^{1-\alpha}& \cdots & c_m \theta_2^{m-\alpha}\\
                                 %
               \vdots & \vdots&                                                 & \vdots\\
               c_1 \theta_m^{1-\alpha} &
               c_2 \theta_m^{1-\alpha}&
               \cdots
               & c_m \theta_m^{m-\alpha}},&
       W
        =
       \pmtrx{\theta_1 & \theta_1^2 & \cdots &\theta_1^m\\
              \theta_2 & \theta_2^2 & \cdots &\theta_2^m\\
              \vdots\ &\vdots\ && \vdots\ \\
              \theta_m & \theta_m^2 &\cdots & \theta_m^m}.
    \end{align*}
    The eigenvalues of the matrix $M=D_1 WD_2W^{-1}$ are the roots of
  the  characteristic polynomial of the form
    \begin{equation}\label{eq:poly}
      \det(M_\alpha-\lambda W)=:\sum_{j=0}^m (-\lambda)^j a_j,
    \end{equation}
    while to show that no eigenvalue of the matrix $M=D_1 WD_2W^{-1}$ is real negative,
    it suffices to check that $a_j>0$ $\forall\,j=0\new{,}\ldots,m$.

  Note that the coefficients $\{a_j\}$ in  \eqref{eq:poly} can be represented via determinants of certain $m\times m$ matrices, denoted
  $M_{\mathcal I}$, where $\mathcal I$ is a subset of  $\{1,\ldots,m\}$, constructed column by column, using the corresponding $k$th column $M_\alpha^k$ of $M_\alpha$ if $k\not\in \mathcal I$, or the $k$th column $W^k$ of $W$ if $k\in \mathcal I$:
   \[
      M_{\mathcal I}^{k} := \begin{cases}
                    M_\alpha^k & \mbox{if}\;\;k\not\in {\mathcal I}\\
                    W^k &\mbox{if}\;\; k\in \mathcal I
                  \end{cases}
                  \qquad\mbox{for any}\;\;
                  {\mathcal I}\subseteq \{1,\ldots, m\},\quad k=1,\ldots, m.
    \]
  For example, for $a_0$ and $a_m$ we immediately have:
      \[
      a_m = \det W = \det M_{\{1,\ldots,m\}}
      \quad\text{and}\quad
      a_0 = \det M_\alpha = \det M_{\emptyset} .
    \]

For the other coefficients, note that if the $k$th column $A^k$ of some matrix $A$ allows a representation
$A^k=B^k-\lambda C^k$ for some column vectors $B^k$ and $C^k$, then
 \begin{align*}
  &\det\Bigl(A^1 \cdots A^{k-1}A^kA^{k+1}\cdots A^m\Bigr)\\&\quad\qquad{}=
   \det\Bigl(A^1 \cdots A^{k-1}B^kA^{k+1}\cdots A^m\Bigr)-\lambda\,
    \det\Bigl(A^1 \cdots A^{k-1}C^kA^{k+1}\cdots A^m\Bigr),
\end{align*}
  which is easily checked using the $k$th column expansion.
  Using this property for all columns in $M_\alpha-\lambda W$, one gets
  \begin{equation}\label{eq:poly2}
  \det(M_\alpha-\lambda W)=
  \sum_{j=0}^m (-\lambda)^{j}\,\Bigl\{\sum_{\#{\mathcal I}=j}\! \det M_{\mathcal I}\Bigr\},
\end{equation}
  where $\#{\mathcal I}$ denotes the number of elements in any  index subset
  ${\mathcal I}\subseteq\{1,\ldots, m\}$.
  (For $m=2, 3$, this is elaborated in Remark~\ref{rem_poly} below.)

  Hence, to show that all coefficients $\{a_j\}$ in  \eqref{eq:poly} are positive, and, thus, to complete the proof, it remains to check that $\det M_{\mathcal I}>0$ for any index subset ${\mathcal I}$.
  For the latter, recalling the columns $M_\alpha$ and $W$, and the construction of $M_{\mathcal I}$,
  one concludes that
\begin{equation}\label{M_I_det}
      \det M_{\mathcal I}
        = \Bigl(\prod_{k\not\in {\mathcal I}} c_k \Bigr)\,
            \det
 \pmtrx{\theta_1^{\beta_1} & \theta_1^{\beta_2} & \cdots &\theta_1^{\beta_m}\\
              \theta_2^{\beta_1} & \theta_2^{\beta_2} & \cdots &\theta_2^{\beta_m}\\
              \vdots\ &\vdots\ && \vdots\ \\
              \theta_m^{\beta_1} & \theta_m^{\beta_2} &\cdots & \theta_m^{\beta_m}}
            ,
            \qquad
            \beta_k:= \begin{cases}
                    k-\alpha & \mbox{if}\;\;k\not\in {\mathcal I},\\
                    k &\mbox{if}\;\; k\in \mathcal I.
                  \end{cases}
 \end{equation}
    Recall that, by \eqref{matr3}, all $c_k>0$, so
     the positivity of $\det M_{\mathcal I}$ is equivalent to the positivity
     of the above determinant involving
    $0<\beta_1<\beta_2<\dots<\beta_m$ and $0<\theta_1<\theta_2<\dots<\theta_m$.
    This is, in fact, the determinant of a so-called generalised Vandermonde matrix \cite{Heineman29},
    the positivity of which is established in \cite{RS00, YWZ01}.
    The desired assertion that $\det M_{\mathcal I}>0$ follows, which completes the proof.
  \end{proof}

  \begin{remark}[Characteristic polynomial \eqref{eq:poly} for $m=2,3$]\label{rem_poly}
 To illustrate the representation of type \eqref{eq:poly2} in the above proof,
  note that for $m=2$ it simplifies to
      \begin{align*}
      \det(M_\alpha-\lambda W)
        &= \lambda^2\vmtrx{\theta_1&\theta_1^2\\[1pt]
                      \theta_2&\theta_2^2}
           -\lambda\left( \vmtrx{\theta_1&c_2\theta_1^{2-\alpha}\\[1pt]
                              \theta_2&c_2\theta_2^{2-\alpha}}
                  +\vmtrx{c_1\theta_1^{1-\alpha}&\theta_1^2\\[1pt]
                              c_1\theta_2^{1-\alpha}&\theta_2^2}\right)
                  +\vmtrx{c_1\theta_1^{1-\alpha} & c_2\theta_1^{2-\alpha}\\[1pt]
                              c_1\theta_2^{1-\alpha} & c_2\theta_2^{2-\alpha}}\\
      &=\lambda^2\,\Bigl[\theta_1\theta_2(\theta_2-\theta_1)\Bigr]\\
      &{}-\lambda\Bigl[c_2\,\theta_1\theta_2(\theta_2^{1-\alpha}-\theta_1^{1-\alpha})
      +c_1(\theta_1\theta_2)^{1-\alpha}(\theta_2^{1+\alpha}-\theta_1^{1+\alpha})\Bigr]\\&{}
      +\Bigl[c_1c_2(\theta_1\theta_2)^{1-\alpha}(\theta_2-\theta_1)\Bigr],
    \end{align*}
  so we immediately see the coefficients (in square brackets) are positive for any $\alpha\in(0,1)$ and $0<\theta_1<\theta_2$.

    For $m=3$, the cubic polynomial $\det(M_\alpha-\lambda W)$ in \eqref{eq:poly} reads
    as $(-\lambda)^3a_3+(-\lambda)^2a_2+(-\lambda)a_1+a_0$,
    where
    \begin{align*}
      a_3=
        & \vmtrx{\theta_1&\theta_1^2&\theta_1^3\\[1pt]
                      \theta_2&\theta_2^2&\theta_2^3\\[1pt]
                      \theta_3&\theta_3^2&\theta_3^3},
                      \qquad\quad
a_0=           \vmtrx{c_1\theta_1^{1-\alpha} & c_2\theta_1^{2-\alpha} & c_3\theta_1^{3-\alpha}\\[1pt]
                       c_1\theta_2^{1-\alpha} & c_2\theta_2^{2-\alpha} & c_3\theta_2^{3-\alpha}\\[1pt]
                       c_1\theta_3^{1-\alpha} & c_2\theta_3^{2-\alpha} & c_3\theta_3^{3-\alpha}},                      \\
        a_2=&
              \vmtrx{c_1\theta_1^{1-\alpha}&\theta_1^2&\theta_1^3\\[1pt]
                         c_1\theta_2^{1-\alpha}&\theta_2^2&\theta_2^3\\[1pt]
                         c_1\theta_3^{1-\alpha}&\theta_3^2&\theta_3^3}
             +\vmtrx{\theta_1&c_2\theta_1^{2-\alpha}&\theta_1^3\\[1pt]
                         \theta_2&c_2\theta_2^{2-\alpha}&\theta_2^3\\[1pt]
                         \theta_3&c_2\theta_3^{2-\alpha}&\theta_3^3}
             +\vmtrx{\theta_1&\theta_1^2&c_3\theta_1^{3-\alpha}\\[1pt]
                         \theta_2&\theta_2^2&c_3\theta_2^{3-\alpha}\\[1pt]
                         \theta_3&\theta_3^2&c_3\theta_3^{3-\alpha}},
           \\ a_1=&
              \vmtrx{\theta_1               & c_2\theta_1^{2-\alpha} & c_3\theta_1^{3-\alpha}\\[1pt]
                         \theta_2               & c_2\theta_2^{2-\alpha} & c_3\theta_2^{3-\alpha}\\[1pt]
                         \theta_3               & c_2\theta_3^{2-\alpha} & c_3\theta_3^{3-\alpha}}
            +\vmtrx{c_1\theta_1^{1-\alpha} & \theta_1^2             & c_3\theta_1^{3-\alpha}\\[1pt]
                         c_1\theta_2^{1-\alpha} & \theta_2^2             & c_3\theta_2^{3-\alpha}\\[1pt]
                         c_1\theta_3^{1-\alpha} & \theta_3^2             & c_3\theta_3^{3-\alpha}}
             +\vmtrx{c_1\theta_1^{1-\alpha} & c_2\theta_1^{2-\alpha} & \theta_1^3\\[1pt]
                         c_1\theta_2^{1-\alpha} & c_2\theta_2^{2-\alpha} & \theta_2^3\\[1pt]
                         c_1\theta_3^{1-\alpha} & c_2\theta_3^{2-\alpha} & \theta_3^3}.
    \end{align*}
    After removing common positive factors in all columns and rows, and
    \new{setting
    $\theta_3=1$ (for $\theta_3<1$ see Remark~\ref{rem_theta_1})}, all eight above determinants reduce to  the following general determinant,
     where $\beta_2\in\{1-\alpha,1,1+ \alpha\}$, $\beta_3\in\{2-\alpha,2,2+\alpha\}$, and so $0<\beta_2<\beta_3$:
    \[
      \begin{vmatrix}
        1 & \,\theta_1^{\beta_2} & \,\theta_1^{\beta_3}\\[1pt]
        1 & \theta_2^{\beta_2} & \theta_2^{\beta_3}\\[1pt]
        1 & 1 & 1
      \end{vmatrix}
      = \begin{vmatrix}
          1 & \ \theta_1^{\beta_2}-1 & \ \theta_1^{\beta_3}-1\\[1pt]
          1 & \theta_2^{\beta_2}-1 & \theta_2^{\beta_3}-1\\[1pt]
          1 & 0 & 0
        \end{vmatrix}
      = (1-\theta_1^{\beta_2})(1-\theta_2^{\beta_2})
        \left(
          \frac{1-\theta_2^{\beta_3}}{1-\theta_2^{\beta_2}}-\frac{1-\theta_1^{\beta_3}}{1-\theta_1^{\beta_2}}
        \right).
    \]
    Although the positivity of the above determinant follows from \cite{RS00, YWZ01}, one
    can show this directly by checking that the function $ g(s):=\frac{1-s^{\beta}}{1-s}$, where
    $\beta:=\beta_3/\beta_2>1$, is increasing $\forall\,s\in(0,1)$.
    Note that $g'(s)>0$ is equivalent to $\beta s^{\beta-1}(1-s)<1-s^\beta$, or
    $\hat g(s):=\beta s^{\beta-1}-(\beta-1)s^\beta<1$.
    The latter property follows from $\hat g(1)=1$ combined with
    $\hat g'(s):=\beta(\beta-1) s^{\beta-2}(1-s)>0$ for $s\in(0,1)$.
  \end{remark}

  \begin{remark}[Case $\alpha=1$]
    Setting $\alpha=1$ in \eqref{eq:problem} and throughout our evaluations,
    the discretization \eqref{matrix_col_U} becomes the collocation method for the classical parabolic differential equation
    \[
      \partial_t u +\LL u = f(x,t).
    \]
    Note that in this case \eqref{matr3} yields
    $c_j=\frac{\Gamma(j+1)}{\Gamma(j+1-\alpha)}=j$, while
    \eqref{eq:poly2} and \eqref{M_I_det} remain unchanged.
    However, now
    we only have $\beta_1\leq\beta_2\leq\dots\leq\beta_m$, and if there is one pair of equal powers, the corresponding determinant $\det M_{\mathcal I}$ is zero.
    Nevertheless,
 the sum
$\sum_{\#{\mathcal I}=j}\! \det M_{\mathcal I}$
    in \eqref{eq:poly2}
for each $j=1,\ldots, m$
includes exactly one positive determinant,
which corresponds to ${\mathcal I}=\{m-j+1, \ldots, m\}$ (understood as ${\mathcal I}=\emptyset$
if $j=0$).
Hence, each such sum is positive, so
$a_j>0$ $\forall\, j=0,\ldots,m$, so
Theorem~\ref{the_cor} remains valid for the classical parabolic case $\alpha=1$ with any $m\ge 1$
and any set of collocation points $0<\theta_1 <\theta_2< \dots < \theta_m\new{\leq1}$.
  \end{remark}

\subsection{A semi-computational approach}\label{ssec_comp}
In addition to Theorem~\ref{the_cor}, which guarantees
that the $m\times m$ matrix  $M=D_1 WD_2W^{-1}$ from
 Theorem~\ref{the_M_matrix} has no real negative eigenvalues
(and, hence,
existence and uniqueness of collocation solutions for any $m\ge 1$
 and any set of  collocation points $\{ \theta_\ell\}$), it may be of interest to have more precise  information about the eigenvalues of this matrix (in terms of where they lie on the complex plain).
In most practical situations, one is interested in a particular distribution of the collocations points (such as equidistant points), in which case
 all eigenvalues of $M$ for all $\alpha\in(0,1)$
are easily computable, as we now describe.

  \begin{lstlisting}[float,caption=Compute spectrum of collocation matrix for fractional derivative,  basicstyle=\footnotesize, escapechar=@, label=alg:compEV]
function sigma = spectrum(m,alpha)
% t = (1:m)/m;                            % equidistant points
t  = (cos(pi*(m-1:-1:0)/m)+1)/2;          % Chebyshev
D1 = diag(t.^(-alpha));
W  = (t').^(1:m);
D2 = diag(gamma(2:m+1)./gamma((2:m+1)-alpha));
M  = D1*W*D2;                             % matrix
sigma = eig(M,W);                         % compute all (gen) EV
\end{lstlisting}

  Given the order $m$ of the collocation method and a particular distribution of the collocations points,
  the computation of approximate eigenvalues for the entire range $\alpha\in(0,1)$
  can be easily performed using designated commands typically available in high-level programming languages
  (such as the command \texttt{eig} in Matlab).
  Here we provide a Matlab script, see Script~\ref{alg:compEV}, that computes, for any $\alpha$ and $m$, all eigenvalues of $M=D_1 WD_2W^{-1}$.
  Recall that Theorem~\ref{the_M_matrix} requires that this matrix has no real negative eigenvalues.
  Thus, such a code allows one to immediately check the applicability of the well-posedness result of Theorem~\ref{the_M_matrix} for any fixed $m$ and collocation point set $\{\theta_\ell\}$ of interest (the latter is denoted by \texttt{t} in Script~\ref{alg:compEV}).

  \begin{figure}[bpt]
    \begin{center}
    \ifusetikz
      \input{cheb_2} 
      \input{cheb_3}
      \input{cheb_5}
      \input{cheb_8}
    \else
       \includegraphics{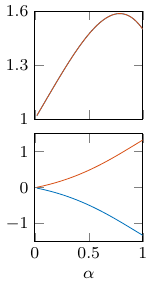}
       \includegraphics{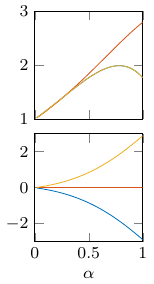}
       \includegraphics{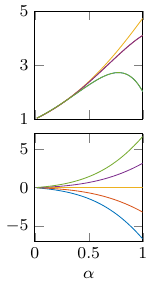}
       \includegraphics{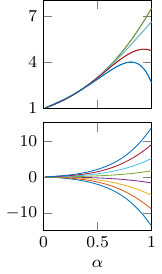}
    \fi
    \end{center}\vspace{-0.6cm}
    \caption{Eigenvalues for collocation methods using Chebyshev points $m\in\{2,3,5,8\}$ (left to right) and real parts (top), imaginary parts (bottom)\label{fig:plotEVcheb}}
  \end{figure}
  \begin{figure}[tp]
    \begin{center}
    \ifusetikz
      \input{uni_2} 
      \input{uni_3}
      \input{uni_5}
      \input{uni_8}
    \else
       \includegraphics{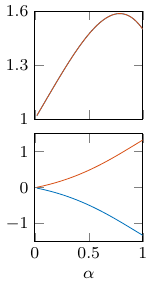}
       \includegraphics{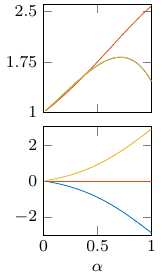}
       \includegraphics{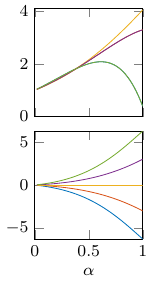}
       \includegraphics{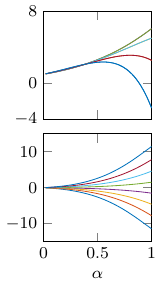}
    \fi
    \end{center}\vspace{-0.6cm}
    \caption{Eigenvalues for collocation methods using equidistant points and $m\in\{2,3,5,8\}$ (left to right) and real parts (top), imaginary parts (bottom)\label{fig:plotEVuni}}
  \end{figure}
  \begin{figure}[htb]
    \begin{center}
    \ifusetikz
      \input{lob_2} 
      \input{lob_3}
      \input{lob_5}
      \input{lob_8}
    \else
       \includegraphics{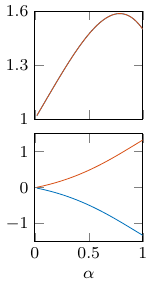}
       \includegraphics{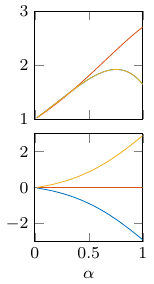}
       \includegraphics{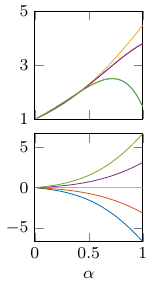}
       \includegraphics{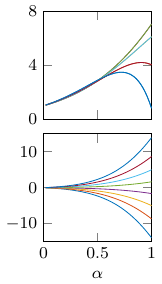}
    \fi
    \end{center}\vspace{-0.6cm}
    \caption{Eigenvalues for collocation methods using Lobatto points and $m\in\{2,3,5,8\}$ (left to right) and real parts (top), imaginary parts (bottom)\label{fig:plotEVlob}}
  \end{figure}

  In Figures~\ref{fig:plotEVcheb}--\ref{fig:plotEVlob} we plot the eigenvalues for $m\in\{2,3,5,8\}$
  with, respectively, Chebyshev, equidistant, and Lobatto distributions of collocation points.

  In addition to the presented graphs, we have numerically investigated the eigenvalues  for $m\le 20$
  and can report the following observations.
  \begin{itemize}
    \item None of the eigenvalues is a real negative number
    whether Chebyshev, Lobatto, or equidistant collocation point sets are used \new{for} $m\le 20$.

    \item If either Chebyshev or Lobatto collocation point sets are used for $m\le 20$, not only
          there are no real negative eigenvalues, but
          the real parts of all eigenvalues are positive.
    %
    \item For equidistant collocation points we observe
    some eigenvalues with negative real parts for some $m$; see, e.g., the case $m=8$ \new{in} Figure~\ref{fig:plotEVuni}.

    \item For odd $m$ one observes exactly one positive eigenvalue, all others are complex-conjugate pairs,
          while for even $m$ we only have complex-conjugate pairs.
  \end{itemize}
  Overall, the above observations confirm the well-posedness of the considered collocation scheme for
  three popular distributions of collocation points for all $m\le 20$.

\section{Semilinear case}\label{sec_semi}

  Consider a semilinear version of~\eqref{eq:problem} of the form
  \begin{equation}\label{eq:problem_sem}
    \pt_t^{\alpha}u+\LL u=f(x,t,u)\qquad
    \mbox{where}\;\; |\partial_uf(\cdot,\cdot,\cdot)|\le\mu\;\;\mbox{in}\;\Omega\times (0,T]\times\R,
  \end{equation}
  for some positive constant $\mu$.
  Without loss of generality, we shall assume that $c(x)=0$ in the definition~\eqref{LL_def} of $\LL$
  (as the term $c(x)u$ can now be included in the semilinear term $f(x,t,u)$).

  For the above semilinear equation, the collocation method of type
  \eqref{Col_method} reads as
  \begin{equation}\label{Col_method_semi}
    (\partial_t^\alpha+\LL)\,U(x, t^\ell_k)=f(x, t_k^\ell, U(x, t^\ell_k))\quad \mbox{for}\;\; x\in\Omega,\;\;  \ell\in\{1,\ldots, m\},\;\;k\ge 1.
  \end{equation}
  Assuming that the above system has a solution, a standard linearization yields
  \[
    f(x, t_k^\ell, U(x, t^\ell_k))=f(x, t_k^\ell, 0)+c(x,t^\ell_k)\,U(x, t^\ell_k),
  \]
  where $|c(x,t^\ell_k)|\le \mu$, so \eqref{Col_method_semi} 
  shows even more resemblance with \eqref{Col_method}:
  \[
    (\partial_t^\alpha+\LL-c(x,t^\ell_k))\,U(x, t^\ell_k)=f(x, t_k^\ell, 0).
  \]
  Hence, recalling Lemma~\ref{lem_matr} and Remark~\ref{rem_matrix_L_t}, we arrive at the following version of~\eqref{matrix_col}
  for $\vec V=\vec V(x)=W^{-1}\vec U(x)$:
  \begin{equation}\label{coersive_semi}
    \bigl(D_1 WD_2+\tau^\alpha\LL W\bigr)\,\vec{V} = \vec{\mathcal{F}}\vec U,
    \qquad\;\;
    \vec U = W \vec{V},
  \end{equation}
 with
  $\vec{\mathcal{F}}\vec U=\vec{F}+\tau^\alpha D_c \vec{U}$,  where
  $\vec{F}=\vec{F}(x)$ remains as in \eqref{matrix_col} except for it is now defined using $f(\cdot,\cdot, 0)$ in place of $f(\cdot,\cdot)$,
  while
  \[
    D_c:={\rm diag}(\hat c(x,\theta_1),\cdots,\hat c(x,\theta_m)),
    \qquad
    \hat c(\cdot, t):=c(\cdot,(t-t_{k-1})/\tau_k).
  \]
  The similarity stops here, as we cannot assume the positivity of $\hat c$, and instead rely on $|\hat c|\le\mu$ (not to mention that $D_c$ also depends on $\vec U$).

  Note also that $\vec{\mathcal{F}}$ 
  satisfies
  $\vec{\mathcal{F}}\vec U-\vec{\mathcal{F}}\vec U^*=\tau^\alpha D_c^* (\vec U-\vec U^*)$
  $\forall\, \vec U,\,\vec U^*\in(L_2(\Omega))^m$,
  where $D_c^*$ is an $m\times m$ diagonal matrix,  all elements of which satisfy $|\hat c^*_j(x)|\le \mu$.

  Hence, a version of  Theorem~\ref{the_LaxM} remains applicable
  under the stronger condition that the matrix
  $W^{\top}D WD_2+(W^{\top}D WD_2)^{\top}$ is positive-definite
  (rather that positive-semidefinite), and the additional condition
  $\mu\tau^\alpha<C_{\mathcal A}$, where
  the constant $C_{\mathcal A}$ is related to the coercivity of the first term in \eqref{AA_from_new}.
  (Note also that now in Corollary~\ref{cor_m_2} (for $m=2$) we need to impose the strict inequality
  $\theta_1<\theta^*=1-\frac12\alpha$.)

  To be more precise,
\new{$\tau$ is assumed} sufficiently small to ensure that
  $\tau^\alpha \mu \langle D_0  W\vec V, W\vec V \rangle$ is strictly dominated by the first term  in \eqref{AA_from_new}.
  \new{Thus, for some $\epsilon>0$ one has $\AA (\vec V, \vec V)\ge(1+\epsilon)\,\tau^\alpha \mu \langle D_0  \vec U, \vec U \rangle$.
  Then
  one can show} that
  the operator $(D_1 WD_2W^{-1}+\tau^\alpha\LL )^{-1}\vec{\mathcal{F}}$ is a strict contraction on $(L_2(\Omega))^m$
  (in the norm induced by the $m\times m$  matrix $D_0$), so Banach's fixed point theorem
  \cite[section 9.2.1]{evans} yields
  existence and uniqueness of a solution $\vec U$ (for the regularity of which one can use the original  Theorem~\ref{the_LaxM}).

  For the applicability of Theorem~\ref{the_M_matrix}, rewrite \eqref{coersive_semi}
  as $(M+\tau^\alpha\LL)\vec U=\vec{\mathcal{F}}\vec U$,
  where $M=D_1 WD_2W^{-1}$.
  Now, in view of \eqref{resolvent}, \mbox{$\tau^\alpha \mu C_M<1$} guarantees that the operator $(M+\tau^\alpha\LL)^{-1}\vec{\mathcal{F}}=R(\tau^\alpha\LL\,;M)\vec{\mathcal{F}}$ is a strict contraction on $(L_2(\Omega))^m$, so Banach's fixed point theorem \cite[section 9.2.1]{evans} again yields
  existence and uniqueness of a solution $\vec U$ in this case.

\section*{Conclusion}
  We have considered continuous collocation discretizations in time for
  time-fractional parabolic equations with a Caputo time derivative of order $\alpha\in(0,1)$.
  For such discretizations,  sufficient conditions for existence and uniqueness of their solutions
  are obtained using two approaches: the Lax-Milgram Theorem
  and the eigenfunction expansion; see Theorems~\ref{the_LaxM} and~\ref{the_M_matrix}.
The resulting sufficient conditions, which involve certain $m\times m$ matrices,
  are verified
      both
       analytically, for all $m\ge 1$ and all sets of collocation points
       (see Theorem~\ref{the_cor}),  and computationally, for
        three popular distributions of collocation points and all $ m\le 20$.
  Furthermore, the above results are extended to the semilinear case.

\section*{Acknowledgement}

%

The authors are grateful to Prof. Hui Liang of Harbin Institute of Technology, Shenzhen, for raising the question of the existence of collocation solutions,
which initiated the study presented in this paper.

%

   \section*{Declarations}
   \textbf{Conflicts of interest/Competing interests} The authors have no competing interests to declare that are relevant to the content of this article.\\
   \textbf{Data availability}
   The datasets generated during and analysed during the current study are available from the corresponding author on reasonable request.\\
   \textbf{Code availability}
   The software is available from the corresponding author on reasonable request.

  \bibliographystyle{plain}
  \bibliography{lit}

\end{document}